\documentstyle[a4,11pt,amssymb]{article}
\newenvironment{proof}{{\sc Proof of}}{\raisebox{.8ex}{\framebox[2mm][b]{}}}
\newtheorem{theorem}{\large\bf Theorem}

\newtheorem{lemma}{\large\bf Lemma}

\newtheorem{corollary}{\large\bf Corollary}

\def\today{\number\day\space\ifcase\month\or
January\or February\or March\or April\or May\or June\or July\or August\or
September\or October\or November\or December\fi
\space\number\year}
\begin{document}
\title{ Edgeworth expansions in operator form }
\author {Zbigniew S. Szewczak\footnote{Nicolaus Copernicus
University, Faculty of Mathematics and Computer Science, ul.  Chopina
12/18, 87-100 Toru\'n, Poland, e-mail: zssz@mat.uni.torun.pl}}
\date{\today}\maketitle
\begin{abstract}
An operator form of asymptotic expansions for Markov chains is
established. Coefficients are given explicitly. Such expansions
require a certain modification of the classical spectral method. They prove to be
extremely useful within the context of large deviations.

\end{abstract}

\section{ Introduction }

 Let $\{ \xi_k\}_{k\in {\Bbb Z}_{+}}$ be a homogeneous Markov chain
defined on a probability space $(\Omega, {\cal F}, {\rm P})$.
Denote by ${\Bbb S}$ and ${\cal S},$ respectively, the phase-space
and its $\sigma-$algebra of measurable subsets. Further, denote by
$P(x,A),\ x\in{\Bbb S},\ A\in{\cal S}$ the transition probability
kernel of the chain. It means that for each $A\in {\cal S},\
P(x,A)$ is a non-negative measurable function on ${\Bbb S}$ while
for each $x\in{\Bbb S},\ P(x,A)$ is a probability measure on
${\cal S}.$ In what follows we assume that the chain is uniformly
ergodic. So, there exists a stationary distribution denoted by
$\pi.$

Consider the sequence of random variables
$X_0=f(\xi_0),\ldots,X_n=f(\xi_n)$ determined by a measurable
function $f\!\!:{\Bbb S}\to{\Bbb R}.$ In what follows we assume
that
\begin{equation}
\label{sigma}
 \sigma^2={\rm E}_{\pi}[X_0^2] + 2\sum_{n=1}^{\infty} {\rm E}_{\pi}[X_0X_{n}]>0.
\end{equation}

There exists a huge literature concerning the limit theorems for
successive sums $ S_n=\sum_{i=1}^n X_i,\quad n=1,2,\ldots$. For
our purposes, it is enough to keep in mind only the works of S.
Nagaev (1957) and (1961) and the monograph by Sirazhdinov and Formanov (1979).
Despite the theory of limit theorems is well developed, some
settings seem to be set aside. For example, in Szewczak (2005)
it was shown that the Cram\'er method of conjugate distributions
assumes a special form of the local limit theorem that was not
considered before. The case studied in Szewczak (2005) concerns Markov
chains with a finite number of states. It worth noting that the
large deviation theorems, established there, proved to be very
useful in statistics of Markov chains (see A. Nagaev (2001) and (2002)).
The mentioned form of the local limit theorem means
the weak convergence of the measures
\begin{equation}\label{ea1}
Q_x^{(n)}(A\times B) = \sigma\sqrt{2\pi n}{\rm P}_{x}^{(n)}
\left[\,X_1+\ldots+X_n\in A,\,\xi_n\in B\,\right],
\end{equation}
where
\[{\rm P}_{x}^{(n)}(\xi_1\in A_1,\dots,\xi_n\in A_n)=\int_{A_1}\! P(x,{\rm
d}x_1)\int_{A_2}\! P(x_1,{\rm d}x_2)\dots\int_{A_n}\!
P(x_{n-1},{\rm d}x_n),\] $A_k\in{\cal S},\ k=1,\dots,n,\ B\in{\cal
S},\ A\in{\cal B}({\Bbb R})$ and $x\in{\Bbb S}.$

Define the linear operators

\begin{equation}\label{ea2}
({\bf K}_n g)(x)  = \int P(x, {\rm d}x_1) \cdots \int
P(x_{n-1},{\rm d}x_n) g(x_n) K_n(x_1,\ldots, x_n);\quad  g\in
L^{\infty}(\mu),
\end{equation}
where  $K_n,$ $n=1,2,\ldots,$ are measurable kernels, and
$L^{\infty}(\mu)$ is the Banach space of measurable functions
equipped with the  essential supremum norm
\[ \bigl\bracevert g \bigl\bracevert ={\rm ess}\sup_x\vert g(x)\vert =
\inf \{a;\, \mu\{x;\, \vert g(x)\vert > a \} = 0\}, \]  $\mu$ is
the initial distribution, i.e. $\mu(A)={\rm P}[\,\xi_0\in A\,],\,$
$A\in{\cal S}.$

Various probability  measures of interest can be represented as a
set  indexed family of the operators (\ref{ea2}). If one  puts
\[K_{n,A}(x_1,\ldots,x_n)= \sigma\sqrt{2\pi
n}I_A(f(x_1)+\cdots+f(x_n)),\quad g(x_n)= I_B(x_n),\, A\in{\cal
B}({\mathbb R}),
\]
then (\ref{ea1}) takes the form
\begin{equation}
\label{ea3}
Q_x^{(n)}(A\times B) = ({\bf K}_{n,A}I_B)(x).
\end{equation}
Similarly,
\begin{equation}
\label{ea4} {\rm P}_{x}[\frac{X_1+\cdots+X_n}{\sigma\sqrt{n}}\in
A, \xi_n\in B\,] = ({\bf K}_{n,A}g)(x)
\end{equation}
provided
\[K_{n,A}(x_1,\ldots,x_n)
=
I_A\left(\frac{f(x_1)+\cdots+f(x_n)}{\sigma\sqrt{n}}\right),\quad
g(x_n)= I_B(x_n).
\]

When $x$ is fixed  the weak convergence of the measures
(\ref{ea3}) (or (\ref{ea4}))  means a form of the classical local
limit (or central limit theorem). Naturally, we expect that the
measures (\ref{ea3}) weakly converge to $\lambda\times\pi$ while
(\ref{ea4}) converge to $\nu\times\pi$ where $\lambda$ is the
Lebesgue measure on ${\cal B}({\mathbb R})$ and
\[\nu(A)={{1}\over{\sqrt{2\pi}}}
\int_Ae^{-{{u^2}\over{2}}}{\rm d}u.\] Such statement can be
embedded into the following scheme of convergence.

Define
\begin{equation}
\label{epc}
\Vert {\bf K}\Vert_{+}=
\sup_{\{g\geq 0;\,g\in L^{\infty}(\mu), \,\bigl\bracevert
g\bigl\bracevert\leq 1\}}\!\!\!\! \bigl\bracevert {\bf K}g\bigl\bracevert.
\end{equation}
Consider a  family of sequences $\{{\bf K}_{n,A}\},$ $A\in{\cal
A}\subset {\cal B}(\mathbb R).$ We say that a sequence ${\bf
K}_{n,A}$ is $L_{\cal A}^{\infty}(\mu)$-strongly convergent to
${\bf K}_A$ if
\begin{equation}
\label{da2}
\sup_{A\in{\cal A}}\Vert {\bf K}_{n,A} -  {\bf K}_A\Vert_{+}\to 0
\qquad\mbox{as}\qquad n\to\infty.
\end{equation}
Let ${\cal A}=\{z\in{\mathbb R}\,\vert\,(-\infty,z\sigma)\}.$ If
the sequence of operators is defined as in (\ref{ea4}) then the
limit operator in (\ref{da2}) has the form
\[({\bf K}_Ag)(x)=\psi(x)\nu(A)\int_{\Bbb S}g(s)\mu( {\rm d} s)\]
where $\psi(x)\equiv 1.$ This fact is formally more general than
e.g. Th. 2.2 in Nagaev (1957) though its proof does not require
serious efforts. It is of much greater interest to establish the
{\em operator form} of the asymptotic expansions for the sequence
$\{{\mathbf K}_{n,z}\}$ determined by the kernels
\[K_{n,z}(x_1,\ldots,x_n) = I_{(-\infty, z)}
\left({{f(x_1)+\cdots+f(x_n)}\over{\sigma\sqrt{n}}}\right),\
z\in\Bbb R.\]

Such asymptotic expansions is basic goal of the present paper. The
paper is organized as follows.  In Section 2 the main results are
stated. In Section 3 a new estimate for the so-called
characteristic operator in the neighborhood of zero is established
(Cf. Lemma 1.6 in Nagaev (1961)).
The proofs are given  in Section 4.

\section{ The main results}

\noindent

In order to state the main results of the paper we have to
introduce the indispensable notation. We are going to establish an
asymptotic expansion  of the form
\begin{equation}
\label{eb1}
\Vert {\bf K}_{n,z} -\sum_{m=0}^{k-2} n^{-{{m}\over{2}}}{\bf A}_{m,z}
\Vert_{+} = \, o(n^{-{{k-2}\over{2}}}),
\end{equation}
where ${\bf A}_{m,z}$ are linear operators defined on
$L^{\infty}(\mu),\ m=0,1,\ldots,\ z\in {\mathbb R}.$ The operators
${\bf A}_{m,z}$ are expressed through the Hermite polynomials
$H_k$ and certain derivatives  of the  so-called characteristic
operator
\[ \hat{\bold P}(\theta)(g)(x) = \int e^{i\theta f(y)} g(y) P(x, {\rm d}y), \]
where $g\in L^{\infty}(\mu).$ More precisely, let
$\lambda(\theta)$ be the principal eigenvalue of $\hat{\bold P}(\theta)$
and $\hat{\bold P}_1(\theta)$ be the projection on the
eigenspace corresponding to $\lambda(\theta).$ Assume that
$\hat{\bf P}(\theta)$ is $k$-times strongly differentiable at
$\theta=0$ and ${\bf P}=\hat{\bf P}(0)$ is $L^{\infty}$-regular
(or primitive),  i.e. there exist $C>0$ and $\gamma,\,$
$0\leq\vert\gamma\vert<1,\,$ such that
\begin{equation}
\label{ae81}
\bigl\bracevert {\bold P}^n g - {\bold \Pi} g \bigl\bracevert \leq C
\vert\gamma\vert^n \bigl\bracevert g \bigl\bracevert,\qquad\qquad
g\in L^{\infty}(\mu),
\end{equation}
where \[({\bold \Pi}g)(x) = \psi(x)\int_{\Bbb S}g(s)\mu( {\rm d}
s)=\psi(x){\rm E}_{\pi}[g]\] (Cf. Gudynas (2000)).
Then $\hat{\bf P}_1(\theta)$ and $\ln{\lambda}(\theta)$ admit the following
MacLaurin expansions:
\[ \hat{\bold P}_1(\theta) = \sum_{m=0}^k {{(i\theta)^m}\over{m!}}
\hat{\bold P}_1^{(m)} +
o(\vert\theta\vert^k),\qquad\mbox{and}\qquad \ln{\lambda(\theta)}
=  \sum_{m=0}^k {{(i\theta)^m}\over{m!}}\gamma_m +
o(\vert\theta\vert^k).
\]
Here, the operators $\hat{\bold P}_1^{(m)}$ can be explicitly
expressed in terms of ${\bold P}$ and ${\bold \Pi}$ (see Lemma
\ref{al6}). The coefficients  $\gamma_m,$ $m=0,1,\ldots$ are
called {\em cumulants}. In what follows we assume $\gamma_1={\rm
E}_{\pi}[f]=0$ thus $\gamma_2=\sigma^2,$ where $\sigma^2$ is
defined by (\ref{sigma})
and $\gamma_3=\mu_3$ is defined
in Lemma 1.2 in Nagaev (1961).
Let $\frak N$ and $\frak n$ denote
the distribution function and the density function of the standard
normal law. Introduce the operators defined on
$L^{\infty}(\mu):\,$ ${\bf A}_{0,z} = {\mathfrak N}(z){\bf \Pi},\,
{\bf A}_{\nu,z}=\sum_{j=0}^{\nu} a_j(z){\hat{\bold P}_1^{(j)}},$
where
\[a_j(z) =- {\mathfrak n}(z)\!\!\!\!\!
\sum_{(k_1,k_2,\ldots,k_{\nu-j})\in {\cal K}_{\nu-j}} \!\!\!\!\!
a_{j,\nu-j}H_{\nu-1+2\sum\limits_{i=1}^{\nu-j} k_i}(z),\quad
a_{\nu}= - {\mathfrak n}(z)H_{\nu-1},
\]
\[a_{j,\nu-j} = {{1}\over{{j}!\sigma^{j}}}
\prod_{m=1}^{\nu-j} {{1}\over{k_m!}}
\Bigl({{\gamma_{m+2}}\over{(m+2)!\sigma^{m+2}}}\Bigl)^{k_m}\!\!\!,
\]
and ${\cal K}_m = \{(k_1,\ldots,k_m)\,;\, \sum_{i=1}^{m} ik_i=
m,\, k_i\geq 0, i=1,\ldots,m \}$. Thus, the operators
${\bf A}_{\nu,z}$ are well-defined provided $\hat{\bf P}(\theta)$ is
$k$-times strongly differentiable at $\theta=0$ and $\sigma>0$.

Let $r(\theta)$ be the spectral radius of $\hat{\bf P}(\theta)$.
It is well known that $r(\theta)$ inherits many principal
properties of the characteristic functions. In order to establish
asymptotic expansions (\ref{eb1})  we have to assume that
\begin{equation}
\label{eslat}
r(\theta)<1,\, \theta\neq 0,\qquad\mbox{and}\qquad
\limsup_{\vert\theta\vert\to\infty}  r(\theta)<1.
\end{equation}

The second inequality in (\ref{eslat}) is analogous to the
well-known Cram\'er condition (C). As to the first one it
guarantees that the distributions of $\sum_{i=1}^n X_i$ for all
sufficiently large $n$ is non-lattice.

The operator form of asymptotic expansions implies such properties
of the considered Markov chain as strong differentiability of
$\hat{\bf P}(\theta),$ primitiveness  and  (\ref{eslat}). Of
course, one could simply assume that these properties take place.
Another way is to give a simply verified condition that guarantees
these properties. As such  we  take the following \vskip
2pt\noindent
{Condition $(\Psi)$}:\\
{\it there exist $\alpha>0$ and $\beta<\infty$ such that for every
Borel set $A$ of a positive measure $\mu$ we have
$\alpha\mu(A)\leq P(x,A) \leq \beta\mu(A)$  for $\mu$-a.a.
$\,x\in{\Bbb S}.$}

\vskip 2pt\noindent

This condition enables us to verify the required properties by the
initial distribution $\mu$. For example if Condition $(\Psi)$ is
fulfilled then (\ref{eslat}) takes place provided
$\mu_f=\mu\!\circ\! f^{-1}$ is non-lattice and

\begin{equation}
\label{ecra}
\limsup_{\vert\theta\vert\to\infty}\vert
{\widehat\mu_f}(\theta)\vert < 1,
\end{equation}
where $\widehat{\mu}_f=\int e^{i\theta f(y)}\mu({\rm d}y)$.
Moreover, if $\int\vert f(y)\vert^k\mu({\rm d}y)<\infty$ then
$\hat{\bf P}(\theta)$ is $k$-times strongly differentiable. It
should be noted (see the proof of Lemma 3.1 in Jensen (1991))
that Condition $(\Psi)$ implies $\sigma^2=\gamma_2>0$.

Now we, are able to state the main results.

\begin{theorem}\quad
\label{tn2} Let Condition $(\Psi)$ is fulfilled. If $\int \vert
f(x)\vert^k\mu({\rm d}x)<\infty,$ $k>3,$ and $\mu_f$ satisfies
(\ref{ecra}) then (\ref{eb1}) holds.
\end{theorem}

As in the case of asymptotic expansions for i.i.d. variables
(see Gnedenko and Kolmogorov, 1954, \S42,  Th. 2)
the following statement does not
require the condition (\ref{ecra}).

\begin{theorem}\quad
\label{t4} Let Condition $(\Psi)$ is fulfilled. If $\int \vert
f(x)\vert^3\mu({\rm d}x)<\infty$ and $\mu_f$ is non-lattice then
(\ref{eb1}) holds with $k=3$.
\end{theorem}
In order to clarify the  specificity of the limit theorems given
in the operator form consider two examples. First, let $\{\xi_k
\}$ be a finite state Markov chain,  i.e. ${\mathbb
S}=\{1,\ldots,d\},$ $d\geq 3$. Denote by ${\mathbf P}$ the
transition matrix. The entries of ${\mathbf P}^{\nu}, \nu\geq 0,$
we denote by $p_{ij}^{(\nu)}, i,j\in{\mathbb S},$
$p_{ij}^{(0)}=\delta_{ij}$. For a real function $f$ on ${\mathbb
S}$ define the  matrix ${\bold P}^{(1)}$ with the elements
$f(j)p_{ij},$ $i,j\in{\mathbb S}$. The following statement is of
independent interest.
\begin{corollary}\quad
\label{c1}
Suppose that transition matrix ${\mathbf P}$ is strictly positive.
If $f(\xi_0)$ is non-lattice and $\sum_{k=1}^d  \pi_k f(k)=0$
then uniformly in $z\in\mathbb R$ the matrix
\[
({\rm P}[S_n< z\sigma\sqrt{n}\,;\,
\xi_n=j\,\vert\, \xi_0=i])_{i,j\in{\mathbb S}}
\]
is approximated by the matrix
\begin{equation}
\label{ee3}
{\mathfrak N}(z){\bold \Pi} + n^{-1/2}{\mathfrak n}(z)
({{{\mu}_3}\over{6{\sigma}^3}}(1-{z^2}){\mathbf \Pi}
- {{1}\over{\sigma}}\sum_{\nu\geq 0}
{{\bold \Pi}{\bold P}^{(1)}({\bold P}^{\nu} - {\bold \Pi})
+ ({\bold P}^{\nu} - {\bold \Pi})
{\bold P}^{(1)}{\bold \Pi}})
\end{equation}
with an error $o(n^{-1/2})$. Here,
\[{\bold \Pi}=\left(\begin{array}{llll}
\pi_1 & \pi_2 & \dots & \pi_d\\&\\
\dots & \dots & \dots & \dots\\&\\
\pi_1 & \pi_2 & \dots & \pi_d
\end{array}\right).\]
\end{corollary}
Another particular case of independent interest is covered by the
following statement.
\begin{corollary}\quad
\label{c2} Let ${\mathbb S}=[0,\ 1]$. Suppose that the transition
density $p(x,y)$ is such that $0<p_-\leq p(x,y)\leq p_+<\infty.$
If  $f(\xi_0)$ is non-lattice and $\int f(u)\pi({\rm d}u)=0$ then
the linear operator (\ref{ee3}) $L^{\infty}_{\cal A}$-strongly
approximates the operator
$g\mapsto {\rm E}_{x}[I_{[S_n<z\sigma\sqrt{n}]}g(\xi_n)\,]$ with
an error $o(n^{-1/2})$. Here,
\[({\bold \Pi}g)(x)=\int_{\Bbb S}g(s)\mu( {\rm d} s)\psi(x).\]
\end{corollary}
Note that the classical scalar form of the presented statement is:
\begin{equation}
\label{esae}
{\rm P}_{\pi}[S_n< z\sigma\sqrt{n}\,] - {\mathfrak N}(z)
= n^{-1/2}{\mathfrak n}(z){{{\mu}_3}\over{6{\sigma}^3}}(1-{z^2})
+ o(n^{-1/2})
\end{equation}
(see e.g. Th. 2 in Nagaev (1961)).
The corollaries show that the
operator form of asymptotic expansions is much more sensitive to
the initial conditions than the scalar one. It should be
emphasized that the spectral method suggested by S. Nagaev
(see e.g. Nagaev, 1957)
remains efficient under this new setting though
requires some modification. Furthermore, the cumbersome
calculations, that are typical for asymptotic expansions, can be
implemented using the package {\em Maple.} This power software
proved to be very efficient for such purposes.

\section{ Characteristic operator}

\noindent

Given $m\in\Bbb N$
let us define operator ${\bf P}^{(m)}g = {\bf P}f^mg$.
The following lemma is an extension of the well-known result due to S. Nagaev
(Cf. Nagaev, 1961, pp. 71--75).
\begin{lemma}\quad
\label{al8}
Suppose that (\ref{ae81}) holds
and ${\bold P}^{(1)}$ is a bounded endomorphism.
Then there exists $\xi = \xi(C,\vert\gamma\vert,\Vert{\bold P}^{(1)}\Vert)$
such that for $\vert\theta\vert<\xi$,
\begin{equation}
\label{aae8}
\hat{\bold P}^n(\theta) =
\lambda^n(\theta)\hat{\bold P}_1(\theta)
+ \hat{\bold Q}_n(\theta) + ({\bold P}^n - {\bold \Pi})
\end{equation}
and $\vert\lambda(\theta)-1\vert<\delta$,
where
$\Vert \hat{\bold P}_1(\theta) - {\bold \Pi}\Vert = O(\mid\theta\mid),\,
\Vert\hat{\bold Q}_n(\theta)\Vert
= O(\kappa^n
\mid\theta\mid),$
$ \kappa = {1\over{3}} + {2\over{3}}\vert\gamma\vert,$
$\>\delta = {1\over{3}} - {1\over{3}}\vert\gamma\vert.$
\end{lemma}
\begin{proof} {\sc Lemma \ref{al8}}

Write
${\Gamma}_0 = \{\vert\zeta\vert = \kappa \},\,
{\Gamma}_1 = \{ \vert\zeta - 1 \vert = \delta \} $
and $D=\{\vert\zeta\vert\geq\kappa\}\cap\{\vert\zeta-1\vert\geq\delta\},$
where $\zeta\in \Bbb C$.
Denote by $\hat{\bold R}(\zeta, \theta)$ the resolvent
of $\hat{\bold P}(\theta)$ and set
${\bold R}(\zeta)=\hat{\bold R}(\zeta, 0)$.
Let $\xi = {{1}\over{2\Vert{\bold P}^{(1)}\Vert}}
\left({{1-\vert\gamma\vert}\over{3(3+C)}}\right)^2.$
Consequently for $\mid\theta\mid<\xi$, $\zeta\in D$
(see \S1 in Nagaev (1961)) we may define the projections
\begin{equation}
\label{esp} \hat{\bold P}_1(\theta) = {1\over{2\pi i}}
\oint_{\Gamma_1} \hat{\bold R}(\zeta, \theta) {\rm d}\zeta, \qquad
\hat{\bold P}_2(\theta) = {1\over{2\pi i}} \oint_{\Gamma_0}
\hat{\bold R}(\zeta, \theta) {\rm d}\zeta.
\end{equation}
Thus (\ref{aae8}) holds with
$\hat{\bold Q}_n(\theta) =
\hat{\bold P}^n(\theta)\hat{\bold P}_2(\theta)
- ({\bold P}^n - {\bold \Pi}). $
We see at once that
\[
\hat{\bold P}(\theta)\hat{\bold R}(\zeta, \theta) =
 -{\bold I} + \zeta\hat{\bold R}(\zeta, \theta)
\]
therefore,
\[
\hat{\bold P}^n(\theta)\hat{\bold R}(\zeta, \theta) =
-\sum_{k=1}^n (\hat{\bold P}(\theta))^{n-k}{\zeta}^{k-1} +
{\zeta}^n\hat{\bold R}(\zeta, \theta).
\]
So it easily seen that
\begin{eqnarray*}
\Vert\hat{\bold Q}_n(\theta)\Vert
& = &
\Vert {1\over{2\pi i}}\oint_{\Gamma_0}
\zeta^n(\hat{\bold R}(\zeta, \theta) - {\bold R}(\zeta)) {\rm d}\zeta \Vert \\
& \leq & {1\over{2\pi}}\int_0^{2\pi}\kappa^n {{ 2(3(3+C))^3 \Vert
\hat{\bold P}(\theta) - {\bold P}\Vert}\over { (1 -
\vert\gamma\vert)^2(6(3+C) - 1 + \vert\gamma\vert)}} \kappa {\rm
d}\phi = O(\kappa^n \mid\theta\mid).
\end{eqnarray*}
Similarly we have
$\Vert \hat{\bold P}_1(\theta) - {\bold  \Pi}\Vert = O(\mid\theta\mid).$
The proof is completed.
\end{proof}

The following lemma deals with the existence of the ``operator'' moments.
\begin{lemma}\quad
\label{al11}
If
\begin{equation}
\label{ena1} \lim_{L\to\infty} \bigl\bracevert \int_{\vert
f(y)\vert >L} \vert f(y)\vert^k P(x,{\rm d}y) \bigl\bracevert = 0
\end{equation}
then
\begin{equation}
\label{emac}
\hat{\bold P}(\theta) = \sum_{m=0}^k {{(i\theta)^m}\over{m!}}
{\bold P}^{(m)} + o(\vert\theta\vert^k),
\end{equation}
where  ${\bold P}^{(m)}$ are bounded for $0\leq m\leq k$.
\end{lemma}
\begin{proof} {\sc Lemma \ref{al11}}

Indeed, we have
\begin{eqnarray*}
h^{-1}(i^{(k-1)}\hat{\bold P}^{(k-1)}(\theta + h)g
- i^{(k-1)}\hat{\bold P}^{(k-1)}(\theta)g)  -
i^k \int e^{i\theta y}g(y) f^k(y) P(\,\cdot\,, {\rm d}y) \\
 =
i^k \int e^{i\theta f(y)}g(y) f^k(y)\int_0^1 (e^{ihsf(y)} -1){\rm
d}s P(\,\cdot\,, {\rm d}y).
\end{eqnarray*}
Now, choose $L$ be sufficiently large positive number.
Since,
\begin{eqnarray*}
\lefteqn {
\inf \{ K\, ;\, \mu\{x\, ;\, \vert
\int (if(y))^k\int_0^1(e^{ihsf(y)} - 1){\rm d}s P(x,{\rm d}y)\vert > K \} = 0 \} }\\
& & \leq
\inf \{ K\, ;\, \mu\{x\, ;\, \vert
\int_{\vert f(y)\vert\leq L}
 (if(y))^k\int_0^1(e^{ihsf(y)} - 1){\rm d}s P(x,{\rm d}y)\vert > K \} = 0 \} \\
& & \, +
\inf \{ K\, ;\, \mu\{x\, ;\, \vert
\int_{\vert f(y)\vert>L}
(if(y))^k\int_0^1(e^{ihsf(y)} - 1){\rm d}s P(x,{\rm d}y)\vert > K \} = 0 \}\\
& & \leq {{1}\over{2}}L^{k+1}\vert h\vert + 2\bigl\bracevert
\int_{\vert f(y)\vert >L} \vert f(y)\vert^k P(x,{\rm d}y)
\bigl\bracevert
\end{eqnarray*}
the lemma follows by the Taylor formula.
\end{proof}

The next lemma presents a series expansion for characteristic
projector.
\begin{lemma}\quad
\label{al6}
If a primitive operator ${\bf P}$ satisfies (\ref{emac}) then
\[ \hat{\bold P}_1(\theta) = \sum_{m=0}^k {{(i\theta)^m}\over{m!}}
\hat{\bold P}_1^{(m)} + o(\vert\theta\vert^k). \]
\end{lemma}
\begin{proof} {\sc Lemma \ref{al6}}

 Put, for short
${\bold E}(\zeta) =
\sum_{n\geq 0}
({\bold P}^n - {\bold \Pi})\zeta^{-n-1},$
and ${\bold E} = {\bold E}(1)$.
In view of (1.10) in Nagaev (1957) and (\ref{emac})
we obtain for $\vert\theta\vert<\xi$
\[ \hat{\bold R}(\zeta, \theta) = {\bold R}(\zeta) +
\sum_{n\geq 1} {\bold R}(\zeta)(\sum_{m = 1}^k {\bold P}^{(m)}
{\bold R}(\zeta){{(i\theta)^m}\over{m!}})^n
+ o(\vert\theta\vert^k). \]
Hence taking in the above coefficient at $i\theta$ and using (\ref{esp})
we get for $k=1$
\begin{eqnarray*}
\hat{\bold P}_1^{(1)}
& = & {1\over {2{\pi}i}}\oint_{\Gamma_1}
({{{\bold \Pi}}\over{\zeta - 1}}+{\bold E}(\zeta)){\bold P}^{(1)}
({{{\bold \Pi}}\over{\zeta - 1}}+{\bold E}(\zeta)) {\rm d}\zeta \\
& = & {\bold \Pi}{\bold P}^{(1)}{\bold \Pi} {1\over
{2{\pi}i}}\oint_{\Gamma_1} {1\over{(\zeta - 1)^2}} {\rm d}\zeta +
{1\over {2{\pi}i}}\oint_{\Gamma_1} {\bold \Pi}{\bold P}^{(1)}
{{\bold E}(\zeta)\over{\zeta - 1}} {\rm d}\zeta
\\& & \quad
+ {1\over {2{\pi}i}}\oint_{\Gamma_1} {{\bold E}(\zeta)\over{\zeta
- 1}}{\bold P}^{(1)} {\bold \Pi} {\rm d}\zeta ={\bold \Pi}{\bold
P}^{(1)}{\bold E} + {\bold E}{\bold P}^{(1)}{\bold \Pi}
\end{eqnarray*}
by Cauchy's integral formula.
For $1 < m\leq k$ arguments are similar. We have
to replace every ${\bold R}(\zeta)$ by
${{{\bold \Pi}}\over{\zeta - 1}}+{\bold E}(\zeta)$ in
\[ \hat{\bold P}_1^{(m)} =
{{m!}\over{2\pi i}} \oint_{\Gamma_1}
\sum_{\nu_1+\nu_2+\ldots+\nu_l=m}\!\!\!\!\!\!\! {\bold
R}(\zeta){{{\bold P}^{(\nu_1)}}\over{\nu_1!}} {\bold
R}(\zeta){{{\bold P}^{(\nu_2)}}\over{\nu_2!}} \cdots {\bold
R}(\zeta){{{\bold P}^{(\nu_l)}}\over{\nu_l!}} {\bold R}(\zeta){\rm
d}\zeta,\,\,\nu_k\geq 1.\]
\end{proof}

Now, we are in a position to represent the principal eigenvalue
of the characteristic operator in a power series.
\begin{lemma}\quad
\label{al13}
If a primitive operator ${\bf P}$ satisfies (\ref{emac}) then
\[ \lambda(\theta) = 1 + {{(i\theta)}\over{1!}}\mu_1 +
{{(i\theta)^2}\over{2!}}\mu_2 +
{{(i\theta)^3}\over{3!}}\mu_3 +\cdots+
{{(i\theta)^k}\over{k!}}\mu_k + o(\vert\theta\vert^k). \]
\end{lemma}
\begin{proof} {\sc Lemma \ref{al13}}

It follows from (\ref{aae8}) that
\begin{equation}
\label{epi1} \pi\hat{\mathbf P}(\theta)\hat{\mathbf
P}_1(\theta)\psi =\lambda(\theta)\pi\hat{\mathbf P}_1(\theta)\psi.
\end{equation}
Denote $\hat\lambda^{(k)}=\pi\hat{\mathbf P}_1^{(k)}\psi.$
By virtue of (\ref{emac}) and Lemma \ref{al6}
\[ \lambda(\theta)\sum_{\nu=0}^k
\hat\lambda^{(\nu)}{{(i\theta)^{\nu}}\over{\nu!}} =
\sum_{m=0}^k
\Bigl(\sum_{\nu=0}^m {{m}\choose{\nu}}
\pi{\mathbf P}^{(\nu)}\hat{\mathbf P}_1^{(m-\nu)}\psi
{{(i\theta)^m}\over{m!}}\Bigl)
 + o(\vert\theta\vert^k). \]
Since $\lambda^{(0)}=\hat\lambda^{(0)}=1,$
$\lambda^{(1)}=\hat\lambda^{(1)}=0,$
and $\lambda^{(k)}$ exists so by the Leibniz formula
\begin{equation}
\label{encor}
\lambda^{(m)}= \mu_m =
\sum_{\nu=1}^m {{m}\choose{\nu}}
\pi{\mathbf P}^{(\nu)}\hat{\mathbf P}_1^{(m-\nu)}\psi
- \sum_{\nu=2}^{m-2} {{m}\choose{\nu}}
\lambda^{(\nu)}\hat\lambda^{(m-\nu)}.
\end{equation}
By (\ref{encor}) $\gamma_2=\lambda^{(2)}$,
for $m>2$ also use the equation (1.13) in Petrov (1996).
\end{proof}

The following theorem is the main result of the present Section.
\begin{theorem}\quad
\label{al3}
If a primitive operator ${\bold P}$ satisfies (\ref{emac}),
$k\geq 3$ and $\sigma^2>0$ then
there exists ${\eta}_k>0$ such that for $T_n = \eta_k{\sigma}\sqrt{n}$
and $\mid\theta\mid \leq T_n$ we have
\begin{eqnarray}
\label{ae7}
\lefteqn {
\Vert \hat{{\bold P}}^n({{\theta}\over{{\sigma}\sqrt{n}}})
- e^{-{{\theta^2}\over{2}}}
\Bigl(
\sum_{m=0}^{k-2}
\sum_{j=0}^{m} {{(i\theta)^j}\over{n^{{m}\over{2}}j!\sigma^j}}
{\frak P}_{m-j}(i\theta) \hat{\bold P}_1^{(j)}
\Bigl)
- ({\bold P}^n - {\bold \Pi})\Vert }   \\
& & \qquad\qquad\qquad\leq
{{o(1)}\over{n^{{k-2}\over{2}}}}(\vert\theta\vert^{k-2}
+ \vert\theta\vert^{k-1} + \vert\theta\vert^{k} + \vert\theta\vert^{3(k-2)})
e^{-{{\theta^2}\over{4}}} + O({{\vert\theta\vert}\over{\sqrt{n}}}
\kappa^n), \nonumber
\end{eqnarray}
where
\[ {\frak P}_\nu(i\theta) =
\sum_{(k_1,k_2,\ldots,k_{\nu})\in {\cal K}_{\nu}}
\prod_{m=1}^{\nu} {{1}\over{k_m!}}
\left({{\gamma_{m+2}(i\theta)^{m+2}}\over{(m+2)!\sigma^{m+2}}}\right)^{k_m}
\!\!\!\!\!\!.\]
\end{theorem}
\begin{proof} {\sc Theorem \ref{al3}}

Let $0<\eta_3\leq\xi$ be such that
$\sup_{\vert\theta\vert\leq\eta_3}
\vert \lambda^{(3)}(\theta) - {\mu}_3 \vert
\leq \sigma^3. $
Put, for short
$T_n = \min\{{{{\sigma}^2}
\over{5({3\over{2}}\vert{\mu}_3\vert+{\sigma}^3)}}, {\eta_3} \}
{\sigma}\sqrt{n}. $
By Taylor's formula for $\mid\theta\mid \leq T_n$ we have
\begin{eqnarray*}
\vert \lambda({{\theta}\over{\sigma\sqrt{n}}})\vert
& \geq &  1 - {{\theta^2}\over{2n}}
- {{\vert\theta\vert^3(\vert\mu_3\vert + \sigma^3 + {1\over{2}}
\vert\mu_3\vert)}\over
{6n^{3\over2}\sigma^3}}
\geq
1 - {{T_n^2}\over{2n}}
- {{T_n^3({3\over{2}}\vert\mu_3\vert + \sigma^3)}\over
{6n^{3\over2}\sigma^3}} \\
& \geq &
1 - {{\sigma^6}\over{50({3\over{2}}\vert\mu_3\vert + \sigma^3)^2}}
- {{\sigma^6}\over{6\cdot 125({3\over{2}}\vert\mu_3\vert
+ \sigma^3)^2}}
> 1 - {2\over50}  = {24\over25}\cdotp
\end{eqnarray*}
Hence for $\mid\theta\mid \leq T_n$ by Taylor's formula and Lemma \ref{al13}
\begin{eqnarray}
\label{aae10}
n\ln\lambda({{\theta}\over{\sigma\sqrt{n}}})
& = &  - {{\theta^2}\over{2}} +
{{(i\theta)^3\mu_3}\over{6\sqrt{n}\sigma^3}}
+ {{(i\theta)^4\mu_4}\over{24n\sigma^4}}
- {{(i\theta)^4}\over{8n}}  + \cdots \\
& & + {{(i\theta)^k}\over{k!n^{{k-2}\over{2}}\sigma^k}}\gamma_k +
{{(i\theta)^k}\over{(k-1)!n^{{k-2}\over{2}}\sigma^k}} \int_0^1
(1-x)^{k-1} W_k({{x\theta}\over{\sigma\sqrt{n}}}){\rm d}x,
\nonumber
\end{eqnarray}
where
$W_k(x) = {{{\partial}^k}\over{\partial y^k}}\ln\lambda(y)\big
\vert_{y=x} - \gamma_k. $
Further, it is evident that we can insert $\eta_k\leq\eta_3$ in
$T_n =
\min\{{{{\sigma}^2}
\over{5({3\over{2}}\vert{\mu}_3\vert+{\sigma}^3)}}, {\eta_k} \}
{\sigma}\sqrt{n},$
such that for  $\vert\theta\vert\leq T_n$ we have
\[
6\sigma^2
({{\vert\theta\vert^2}\over{24n\sigma^4}}\vert\gamma_4\vert +
\cdots +
{{\vert\theta\vert^{k-2}}\over{k!n^{{k-2}\over{2}}\sigma^k}}
(\vert\gamma_k\vert+c_k)) < 7, \] where $c_k = \sup_{x\in [0,1]}
\vert W_k(x)\vert$. Since
$W_k({{x\theta}\over{\sigma\sqrt{n}}})\rightarrow_n 0,$  so by the
Lebesgue dominated convergence theorem we get $\int_0^1
(1-x)^{k-1} W_k({{x\theta}\over{\sigma\sqrt{n}}}){\rm d}x = o(1).$
By virtue of (\ref{aae8}), Lemma \ref{al8} and (\ref{aae10}) we
obtain
\begin{eqnarray*}
\lefteqn {
\Vert \hat{{\bold P}}^n({{\theta}\over{\sigma\sqrt{n}}})
 - e^{ - {{\theta^2}\over{2}} +
{{(i\theta)^3\mu_3}\over{6\sqrt{n}\sigma^3}}
+ {{(i\theta)^4\mu_4}\over{24n\sigma^4}}
- {{(i\theta)^4}\over{8n}} +
\cdots + {{(i\theta)^k}\over{k!n^{{k-2}\over{2}}\sigma^k}}\gamma_k }
\hat{\bold P}_1({{\theta}\over{\sigma\sqrt{n}}})
-({\bold P}^n - {\bold \Pi})\Vert }
\\& & \leq
 e^{ - {{\theta^2}\over{2}} +
\cdots +
{{(i\theta)^k}\over{k!n^{{k-2}\over{2}}\sigma^k}}\gamma_k}
\left\vert \exp\{ {{(i\theta)^k\int_0^1 (1-x)^{k-1}
W_k({{x\theta}\over{\sigma\sqrt{n}}}){\rm d}x }
\over{(k-1)!n^{{k-2}\over{2}}\sigma^k}} \} - 1\right\vert O(1) \\
& & \qquad
+\, O({{\vert\theta\vert}\over{\sigma\sqrt{n}}}\kappa^n).
\end{eqnarray*}
By (\ref{aae10}) and the inequality
$ \vert e^x - 1\vert \leq \vert  x\vert e^{\vert x\vert},$
we find that for $\mid\theta\mid\leq T_n$ we have
\[
\big\vert \exp\{
{{(i\theta)^k}\over{(k-1)!n^{{k-2}\over{2}}\sigma^k}}\!\int_0^1\!
(1-x)^{k-1} W_k({{x\theta}\over{\sigma\sqrt{n}}}) {\rm d}x\} -
1\big\vert \leq
{{o(1)\vert\theta\vert^k}\over{n^{{k-2}\over{2}}}}
\exp\{{{c_k\vert\theta\vert^k}\over{k!n^{{k-2}\over{2}}\sigma^k}}\}\cdotp
\]
Hence,
\begin{eqnarray*}
\lefteqn {
\Vert \hat{{\bold P}}^n({{\theta}\over{\sigma\sqrt{n}}})
 - e^{ - {{\theta^2}\over{2}} +
{{(i\theta)^3\mu_3}\over{6\sqrt{n}\sigma^3}}
+ {{(i\theta)^4\mu_4}\over{24n\sigma^4}}
- {{(i\theta)^4}\over{8n}} + \cdots
+ {{(i\theta)^k}\over{k!n^{{k-2}\over{2}}\sigma^k}}\gamma_k}
\hat{\bold P}_1({{\theta}\over{\sigma\sqrt{n}}})
-({\bold P}^n - {\bold \Pi})
\Vert } \\
& & \leq\exp\{ - {{\theta^2}\over{2}}
 + \cdots
+ {{(i\theta)^k}\over{k!n^{{k-2}\over{2}}\sigma^k}}(\vert\gamma_k\vert + c_k)\}
 {{\vert\theta\vert^k}\over{n^{{k-2}\over{2}}}}o(1)
+ O({{\vert\theta\vert}\over{\sigma\sqrt{n}}}\kappa^n) \\
& & \leq
\exp\{ - {{\theta^2}\over{2}} + {{\theta^2}\over{2}}(
{1\over15}{{{3\over{2}}\vert\mu_3}\vert\over
{{3\over{2}}\vert\mu_3\vert + \sigma^3}} +
{7\over15}{{\sigma^3}\over{{3\over{2}}\vert\mu_3\vert
+ \sigma^3}}) \}
{{\vert\theta\vert^k}\over{n^{{k-2}\over{2}}}}o(1)
+ O({{\vert\theta\vert}\over{\sigma\sqrt{n}}}
\kappa^n) \\
& & \leq
o(1){{\vert\theta\vert^k}\over{n^{{k-2}\over{2}}}}
\exp\{ - {{\theta^2}\over{4}} \}
+ O({{\vert\theta\vert}\over{\sigma\sqrt{n}}}
\kappa^n).
\end{eqnarray*}
Thus expanding
$\exp\{{{(i\theta)^3\mu_3}\over{6\sqrt{n}\sigma^3}}
+ \cdots
+ {{(i\theta)^k}\over{k!n^{{k-2}\over{2}}\sigma^k}}\gamma_k\}$
and using Lemma \ref{al6} and Taylor's formula
for $\hat{\bold P}_1({{\theta}\over{\sigma\sqrt{n}}})$
we obtain (\ref{ae7}).
\end{proof}

The following lemma provides an estimate for the
iterates of characteristic operator (for
the proof see Lemma 1.5 in Nagaev (1961)).
\begin{lemma}\quad
\label{lna2}
Let Condition $(\Psi)$ is fulfilled.
Then for $n\geq 1$ and
$\bigl\bracevert g\bigl\bracevert\leq 1$
\[
\bigl\bracevert \hat{\bold P}^n(\theta)g\bigl\bracevert
\leq \left(\sqrt{1-{{\alpha^4}\over{2\beta}}
(1-\vert{\widehat\mu_f}(\theta)\vert^2)}\right)^{n-1}
\!\!\!\!\!\!\!\!\!. \]
\end{lemma}

\section{ Proofs }
\begin{proof} {\sc Theorems \ref{tn2} and \ref{t4}}

By virtue of Condition $(\Psi)$ and \S1 in Nagaev (1957)
${\bold P}$ is primitive in $L^{\infty}(\mu)$
(alternatively one can use Proposition 3.13 in Wu (2000)).
Moreover, it follows also that
\[ \lim_{L\to\infty}\bigl\bracevert
\int_{\vert f(y)\vert >L} \vert f(y)\vert^k P(x,{\rm d}y)
\bigl\bracevert \leq \beta\lim_{L\to\infty} \int_{\vert f(y)\vert
>L} \vert f(y)\vert^k \mu({\rm d}y) = 0
\]
so that (\ref{ena1}) holds.
Write
$F_{gn}(z) = F_{g(\cdot),n}(z) =
({\bf K}_{n,z}g)(\cdot),$ and
${G_{gn}}(z) = \sum_{m=0}^{k-2} n^{-{{m}\over{2}}} {\bf A}_{m,z}g$.
Let $K_{gn}(z)$ be the distribution function that assigns the
mass $({\bold P}^n - {\bold \Pi})g(\cdot)$ at $0$.
Put,
\[  H_{gn}(z) = G_{gn}(z) + K_{gn}(z),\>\,
\hat H_{gn}(\theta) = \int e^{i\theta x}{\rm d}H_{gn}(x), \>\,
\hat F_{gn}(\theta) = \int e^{i\theta x}{\rm d}F_{gn}(x).
\]
Note that
$\hat F_{gn}(\theta) = \hat{\bold P}^n({{\theta}\over{\sigma\sqrt{n}}})(g)$
and
\[\hat G_{gn}(\theta) =
\int e^{i\theta x}{\rm d}G_{gn}(x) = e^{-{{\theta^2}\over{2}}}
\Bigl(\sum_{m=0}^{k-2}{{1}\over{(\sqrt{n})^m}}\sum_{j=0}^{m}
{{1}\over{j!}} \big({{i\theta}\over{\sigma}}\big)^j{\frak
P}_{m-j}(i\theta) \hat{\bold P}_1^{(j)}g(\cdot) \Bigl)\cdot\]
Because of
\[ \bigl\bracevert H_{gn}(z+y) - H_{gn}(z)\bigl\bracevert \leq
\bigl\bracevert  G_{gn}(z+y) - G_{gn}(z)\bigl\bracevert
+ \bigl\bracevert ({\bold P}^n - {\bold \Pi})g\bigl\bracevert, \]
\[ G_{gn}(z+y) - G_{gn}(z)
= y{{\partial }\over{\partial z}}G_{gn}(z) + {\rm sgn}(y)
\int\limits_{{{y-\vert y\vert}\over{2}}}^{{{y+\vert
y\vert}\over{2}}} ({{\partial}\over{\partial u}} G_{gn}(z+u) -
{{\partial }\over{\partial z}}G_{gn}(z)){\rm d}u \]
thus in view
of Th. 5.3 on pp. 146--147 in Petrov (1996)
and (\ref{ae81}) we have
\begin{eqnarray}
\label{ae41}
\bigl\bracevert F_{gn}(z) - G_{gn}(z)\bigl\bracevert
& \leq & \bigl\bracevert F_{gn}(z) - G_{gn}(z)
- K_{gn}(z) \bigl\bracevert  +\,  C\vert\gamma\vert^n
\bigl\bracevert g\bigl\bracevert \nonumber \\
&  \leq  &
{1\over{\pi}}\int_{\vert\theta\vert\leq T}
\bigl\bracevert\hat F_{gn}(\theta) - \hat H_{gn}(\theta)\bigl\bracevert
{{{\rm d}\theta}\over{\vert\theta\vert}} \\
& & \,
+ {3c^2({1\over{\pi}})\over{\pi T}}\sup_{z}
\bigl\bracevert{{\partial}\over{\partial z}} G_{gn}(z)\bigl\bracevert
+ C\vert\gamma\vert^n
\big(1 +  {{2c({1\over{\pi}})}\over{\pi}}\big)
\bigl\bracevert g\bigl\bracevert. \nonumber
\end{eqnarray}
Now, since
$\sup_{z} \bigl\bracevert{{\partial}\over{\partial z}}
G_{gn}(z)\bigl\bracevert$
is bounded and  $\vert\gamma\vert<1$
whence by (\ref{ae41}) for $T=n^k, k \geq 4$, we get
\begin{equation}
\label{efg} \bigl\bracevert F_{gn}(z) - G_{gn}(z)\bigl\bracevert
\leq {1\over{\pi}}\int_{\vert\theta\vert\leq n^k}
\bigl\bracevert\hat F_{gn}(\theta) - \hat
H_{gn}(\theta)\bigl\bracevert {{{\rm
d}\theta}\over{\vert\theta\vert}} + o({{\bigl\bracevert
g\bigl\bracevert}\over{n^{{k-2}\over{2}}}})\cdotp
\end{equation}
By virtue of Th. \ref{al3}
\begin{eqnarray}
\label{efh0}
\lefteqn
{ \int_{\vert\theta\vert\leq T_n}
\bigl\bracevert\hat F_{gn}(\theta) - \hat H_{gn}(\theta)\bigl\bracevert
{{{\rm d}\theta}\over{\vert\theta\vert}} } \\
& & \leq {{o(\bigl\bracevert
g\bigl\bracevert)}\over{n^{{k-2}\over{2}}}}
\int_{\vert\theta\vert\leq T_n} (\vert\theta\vert^{k-3} +
\vert\theta\vert^{k-1} + \vert\theta\vert^k +
\vert\theta\vert^{3k-7}) e^{-{{\theta^2}\over{4}}}{\rm d}\theta +
{{T_n^2}\over{\sqrt{n}}}O(\kappa^n) \nonumber\cdotp
\end{eqnarray}
This established, we have to show that
\begin{eqnarray}
\label{efh} \int_{T_n<\vert\theta\vert\leq n^k}
{{\bigl\bracevert\hat F_{gn}(\theta) - \hat
H_{gn}(\theta)\bigl\bracevert} \over{\vert\theta\vert}} {\rm
d}\theta \leq o({{\bigl\bracevert
g\bigl\bracevert}\over{n^{{k-2}\over{2}}}})\cdotp
\end{eqnarray}
For this observe that
\begin{eqnarray*}
\int_{T_n<\vert\theta\vert\leq n^k}\!\!\!\!\!\!\!\!\!
{{\bigl\bracevert\hat
G_{gn}(\theta)\bigl\bracevert}\over{\vert\theta\vert}} {\rm
d}\theta \leq 2\int_{T_n}^{\infty}\!\!\! e^{-{{\theta^2}\over{2}}}
\left\Vert \sum_{m=0}^{k-2} \sum_{j=0}^{m}
{{(i\theta)^j}\over{n^{{m}\over{2}}j!\sigma^j}} {\frak
P}_{m-j}(i\theta) \hat{\bold P}_1^{(j)} \right\Vert
\bigl\bracevert g\bigl\bracevert {{{\rm
d}\theta}\over{\vert\theta\vert}}
\end{eqnarray*}
and that by (\ref{ae81})
\begin{eqnarray*}
\int_{T_n<\vert\theta\vert\leq n^k} \bigl\bracevert \hat
H_{gn}(\theta) - \hat G_{gn}(\theta)\bigl\bracevert {{{\rm
d}\theta}\over{\vert\theta\vert}} \leq 2C_k\vert\gamma\vert^n
\bigl\bracevert g\bigl\bracevert\ln{n} = o({{\bigl\bracevert
g\bigl\bracevert}\over{n^{{k-2}\over{2}}}})\cdotp
\end{eqnarray*}
Further, by Lemma \ref{lna2} and (\ref{ecra})
there exists $\theta_0$ such that for any $\tau > \theta_0$
\[
n^{{k-2}\over{2}}\!\!\!\!\!\!
\int\limits_{T_n<\vert\theta\vert\leq n^k} \bigl\bracevert \hat
F_{gn}(\theta)\bigl\bracevert {{{\rm
d}\theta}\over{\vert\theta\vert}} \leq
n^{{k-2}\over{2}}\!\!\!\!\!\! \int\limits_{\tau\leq\vert
\theta\vert\leq n^k} \bigl\bracevert(\hat{{\bold P}}(\theta))^n
g\bigl\bracevert {{{\rm d}\theta}\over{\vert \theta\vert}} \leq
C_k n^{{k-2}\over{2}}e^{-cn} \bigl\bracevert g\bigl\bracevert
\ln{n}
\] which with the latter inequalities proves (\ref{efh}).
Consequently, the substitution of (\ref{efh0}) and (\ref{efh})
into (\ref{efg}) yields (\ref{eb1}). For the case $k=3$ set $T =
T_nr_n$ and choose a sequence $r_n\rightarrow \infty$ such that we
have
\[\int_{T_n<\vert\theta\vert\leq T_nr_n}
\bigl\bracevert \hat F_{gn}(\theta)\bigl\bracevert {{{\rm
d}\theta}\over{\vert\theta\vert}} =
\int_{T_n<\vert\theta\vert\sigma\sqrt{n}\leq T_n r_n}
\bigl\bracevert\big( \hat{{\bold P}}(\theta) \big)^n
g\bigl\bracevert {{{\rm d}\theta}\over{\vert \theta\vert}}
=\bigl\bracevert g\bigl\bracevert o(n^{-1/2}). \] This completes
the proof.
\end{proof}
\vskip 0pt\noindent
{\bf Acknowledgment.} The author thanks A. Nagaev for his comments
concerning this exposition.

\vskip 15pt\noindent
{\bf References}
\vskip 5pt\noindent
 Gnedenko, B.V., Kolmogorov, A.N., 1954.
Limit Distributions for Sums of Independent Random Variables.
Addison-Wesley, Reading, Mass.
\vskip 2pt\noindent
Gudynas, P., 2000.
Refinements of the Central Limit Theorem for Homogene\-ous Markov
Chains, in: Yu. V. Prokhorov and V. Statulevi\v cius, eds.
Limit Theorems of Probability Theory. Springer, Berlin, pp.
165--183.
\vskip 2pt\noindent
Jensen, J.L., 1991.
Saddlepoint expansions for sums of Markov dependent variables on a
continuous state space.
Probab. Theory Related Fields 89, 181--199.
\vskip 2pt\noindent
Nagaev, A.V., 2001.
An asymptotic formula for the Bayes risk in discriminating between
two Markov chains. J. Appl. Probab. 38A, 131--141.
\vskip 2pt\noindent
Nagaev, A.V., 2002.
An asymptotic formula for the Neyman--Pearson risk in
dicriminating between two Markov chains.
J. Math. Sci. 111, 3592--3600.
\vskip 2pt\noindent
Nagaev, S.V., 1957.
Some limit theorems for stationary Markov chains.
Teor. Veroyatnost. i Primenen.  2, 389--416.
\vskip 2pt\noindent
Nagaev, S.V., 1961.
More exact statements of limit theorems for homogeneous Markov chains.
Teor. Veroyatnost. i Primenen. 6, 67--86.
\vskip 2pt\noindent
Petrov, V.V., 1996.
Limit Theorems of Probability Theory. Sequences of
Indenpendent Random Variables. Oxford Studies in Probability 4,
Oxford.
\vskip 2pt\noindent
Sirazhdinov, S.H., Formanov, S.K., 1979.
Limit Theorems for Sums of Random Vectors Connected in a
Markov Chain. FAN, Tashkent.
\vskip 2pt\noindent
Szewczak, Z.S. 2005.
A remark on large deviation theorem for Markov chain with finite
number of states.
Teor. Veroyatnost. i Primenen. 50 3, 612--622.
\vskip 2pt\noindent
Wu, L.M., 2000.
Uniformly Integrable Operators and Large Deviations for Markov
Processes. J. Funct. Anal. 172, 301--376.
\end{document}